\documentclass[notitlepage,leqno,10pt]{article}
\usepackage{latexsym}
\usepackage{color} 
\usepackage{amsmath}
\usepackage{amsthm}
\usepackage{amsfonts}
\usepackage{amssymb}
\usepackage{graphicx}
\usepackage{enumerate}

\newcommand{\intbar}{\mathop{\int\makebox(-13.5,0){\rule[4pt]{.7em}{0.3pt}}%
\kern-6pt}\nolimits}

\newcommand{\be}{\begin{equation}}
\newcommand{\ee}{\end{equation}}
\newenvironment{pf}{\noindent{\sc Proof}.\enspace}{\rule{2mm}{2mm}\medskip}

\newtheorem{remark}{Remark}[section]

\newcommand{\C}{\mathbb{C}}
\newcommand{\R}{\mathbb{R}}
\newcommand{\Rn}{\mathbb{R}^n}

\newcommand{\Rdue}{\mathbb{R}^2}

\newcommand{\N}{\mathbb{N}}
\renewcommand{\P}{{\cal P}}

\DeclareMathOperator{\spt}{spt}

\DeclareMathOperator{\Scal}{Scal}
\DeclareMathOperator{\Ric}{Ric}

\begin{document}

\author{Andrea Mondino$^{1}$, Stefano Nardulli$^{2}$}

\date{}

\title{Existence of isoperimetric regions in non-compact Riemannian manifolds under Ricci or scalar curvature conditions}

\newtheorem{lem}{Lemma}[section]
\newtheorem{pro}[lem]{Proposition}
\newtheorem{thm}[lem]{Theorem}
\newtheorem{rem}[lem]{Remark}
\newtheorem{cor}[lem]{Corollary}
\newtheorem{df}[lem]{Definition}
\newtheorem{ex}[lem]{Example}
\newtheorem*{Theorem}{Theorem}
\newtheorem*{Lemma}{Lemma}
\newtheorem*{Proposition}{Proposition}
\newtheorem*{claim}{Claim}

\maketitle

\footnotetext[1]{Scuola Normale Superiore, Piazza dei Cavalieri 7, 56126 Pisa, Italy, E-mail address: andrea.mondino@sns.it}
\footnotetext[2]{Departamento de Matem\'atica - Instituto de Matem\'atica - UFRJ Av. Athos da Silveira Ramos 149, Centro de Tecnologia - Bloco C Cidade Universit‡ria - Ilha do Fund‹o. Caixa Postal 68530 21941-909 Rio de Janeiro - RJ - Brasil, E-mail address: nardulli@im.ufrj.br}
\

\

\begin{center}
\noindent {\sc abstract}. We prove existence of isoperimetric regions for every volume in non-compact Riemannian $n$-manifolds $(M,g)$, $n\geq 2$, having Ricci curvature $\Ric_g\geq (n-1) k_0 g$ and being locally asymptotic to the simply connected space form of constant sectional curvature $k_0$; moreover in case $k_0=0$ we show that the isoperimetric regions are indecomposable. We also discuss some physically and geometrically relevant examples. Finally, under assumptions on the scalar curvature we prove existence of isoperimetric regions of small volume.  
\bigskip\bigskip


\bigskip

\centerline{\bf AMS subject classification: }
49Q10, 53A10, 53C42, 83C99.
\end{center}

\section{Introduction}\label{s:in}
If $(M,g)$ is a \emph{compact} Riemannian $n$-manifold, then standard techniques of geometric measure theory ensure existence of isoperimetric regions (roughly speaking $\Omega\subset M$ is an isoperimetric region if its boundary has least area among the boundaries of regions having the same volume of $\Omega$; for the precise notions see Section \ref{Sec:NP}).

In case $M$ is \emph{non-compact} the question of existence of isoperimetric regions is completely non-trivial and the few known  existence results are quite specific. A simple example where existence fails is the right hyperbolic paraboloid $M_\lambda$ defined by the equation $z=\lambda x y$: here there is no isoperimetric region for any value of the area (see  \cite{Rit1}). More dramatically, it can happen that isoperimetric regions exist just for \emph{some} value of the area (see \cite{CaneteRit} where a complete study of isoperimetry  in the case of quadrics of revolution is performed). Nevertheless there are some cases when the existence of isoperimetric regions for every volume is known: 
\begin{enumerate}
\item $(M,g)$ is complete non-compact but its \emph{isometry group acts co-compactly}  (see \cite{MorgJoh}, \cite{Mor94}, or \cite{GR10} in the context of sub-Riemannian contact manifolds).
\item $(M,g)$ is connected complete non-compact but with \emph{finite volume}  (this is an easy consequence of Theorem 2.1 in \cite{RitRosal}). 
\item The non-compact non simply connected surfaces constructed in \cite{GP}.
\item In several cases when $(M,g)$ is a \emph{cone}, the isoperimetric regions exist for every volume and are characterized (see \cite{MorRit}, \cite{RitRosal}); for \emph{warped products} see \cite{BrMo}. 
\item If $(M,g)$ is a a complete plane with non-negative curvature  (see \cite{Rit}).
\end{enumerate}
The reason for the non-existence of isoperimetric regions for a fixed volume $v>0$ is explained clearly by Theorem 2.1 in \cite{RitRosal} (recalled in Theorem \ref{thm:RR}): the lack of compactness in the variational problem is due to the fact that the minimizing sequences might split into a part  converging nicely to an isoperimetric region, and in another part  of positive volume  going to infinity. The diverging part of the minimizing sequences was studied  by the second author in \cite{NardAJM} using the theory of $C^{m,\alpha}$-pointed convergence of manifolds developed by Petersen \cite{Pet} (see Section \ref{Sec:NP}). In the present paper we adopt this second point of view. 

The main goal of the present work is to add, to the previous list, a class of manifolds admitting isoperimetric regions for all volumes. This is the content of the next theorem.  
\begin{thm}\label{thm:ExMinMain}
Let $(M^n,g)$ be an $n\geq 2$ dimensional complete Riemannian manifold such that
\begin{enumerate}
          \item $(M^n,g)$ is $C^0$-locally asymptotic to the simply connected $n$-dimensional space form of constant sectional curvature $k_0\leq 0$, i.e., for every diverging sequence of points $p_j$ the sequence of pointed manifolds $(M,g,p_j)$ converges in $C^{0}$ topology to $(\mathbb{M}_{{k}_0}^n, x_0)$  ($x_0$ is any point in $\mathbb{M}_{{k}_0}^n$),
          \item $Ric_g\geq (n-1) k_0 g$,
          \item $V(B(p,1))\geq v_0>0$ for every $p \in M$.
\end{enumerate}
Then for every $v>0$ there exists an isoperimetric region $\Omega_v$ of volume $v$ such that
$$\P(\Omega_v)=I_M(v).$$ 
Moreover if $k_0=0$ (i.e. $\Ric_g\geq 0$ and $(M,g)$ is $C^{0}$-locally asymptotically euclidean) then the isoperimetric regions are \emph{indecomposable}.
\end{thm}
Roughly speaking, the last sentence says that the isoperimetric regions are connected if $k_0=0$. For the precise notion of indecomposability see Section \ref{Sec:NP}; see  Definition \ref{def:CmaConv} for the concept of $C^{0}$-pointed convergence of manifolds. 

To our knowledge, this is the first existence result valid for all volumes and all dimensions in the non-compact case under just geometric curvature assumptions and asymptotic conditions on the ambient manifold.

\begin{remark}\label{Examples}
Notice that the class of manifolds satisfying the assumptions of Theorem \ref{thm:ExMinMain} contains many geometrically and physically relevant examples: Eguchi-Hanson and more generally ALE gravitational instantons (these manifolds are the building blocks of the Euclidean quantum gravity theory of Hawking), asymptotically hyperbolic Einstein manifolds (these spaces play a crucial role in the AdS/CFT correspondence in quantum field theory) and Bryant type solitons (which are special but fundamental solutions to the Ricci flow). For a deeper discussion about these spaces see Section \ref{Sec:Examples}.
\end{remark}

In order to prove Theorem \ref{thm:ExMinMain} in Section \ref{Sec:Existence}, in Section \ref{Sec:GeneralProp} we prove some general properties of the isoperimetric regions and of the isoperimetric profile function of a non-compact Riemannian manifold.
Using the results of Section \ref{Sec:GeneralProp} we are also able to perform a finer analysis of the minimizing sequences for the perimeter under the volume constraint in case the manifold has non-negative Ricci tensor, $\Ric\geq 0$: roughly speaking  either they converge to an isoperimetric region or they diverge, but they cannot split into a converging  and a diverging part. For the precise statement see Theorem \ref{thm:Ric>0}.

The previous existence theorem is based on assumptions on the \emph{Ricci} curvature; actually, as the following theorem points out,  if one is interested in the existence of isoperimetric regions of \emph{small volume} it is enough to ask conditions on the \emph{scalar} curvature. 

\begin{thm}\label{thm:ExSmallVol}
Let $(M,g)$ be an $n\geq 2$  dimensional Riemannian manifold of $C^{2,\alpha}$-bounded geometry  and let $S\in \R$. Suppose that $(M,g)$  satisfies the following assumptions:
\begin{enumerate}
          \item for every $\epsilon>0$ there exists a compact subset $K_\epsilon \subset \subset M$ such that the scalar curvature
$$\Scal_g(p)\leq S+\epsilon \quad \forall p\in M\backslash K_\epsilon,$$ 
          \item there exists a point $\bar{p}\in M$ such that $\Scal_g(\bar{p})>S$.
\end{enumerate}
Then there exists a small $v_0>0$ such that for any $0<v\leq v_0$ there exists an isoperimetric region of volume $v_0$. Moreover such an isoperimetric region is a pseudo-bubble having center of mass in a point $\bar{p}_v$ which is converging in Hausdorff distance sense, as $v \to 0$, to the set of points of global maximum of the scalar curvature $\Scal_g$.
 \end{thm}
For the concept of  $C^{2,\alpha}$-bounded geometry see Definition \ref{def:CmaBddGeom}, for the  precise notions of pseudo-bubble and center of mass see Definitions \ref{def:pseudBub} and \ref{def:CenterMass}. 
 
Theorem \ref{thm:ExSmallVol} is also interesting in connection with Theorem \ref{thm:ExMinMain}. Indeed, if the Riemannian manifold $(M,g)$ satisfies the assumptions of Theorem \ref{thm:ExMinMain} and moreover there exists a point $\bar{p}\in M$ where $\Scal_g(\bar{p})>n(n-1)k_0$ then Theorem \ref{thm:ExMinMain} ensures existence of isoperimetric regions for every volume, and Theorem \ref{thm:ExSmallVol} says that these isoperimetric regions, for small volumes, are pseudo-bubbles centered near the points of maximal scalar curvature. For the existence and the characterization of isoperimetric regions of large volume in manifolds which are  asymptotically globally Euclidean  see \cite{MeEi3} (see also \cite{MeEi1} and \cite{MeEi2}). 
\\
 
The article is organized in the following way: in Section \ref{Sec:NP} we recall the notions and the known  results  used throughout the paper,  in Section \ref{Sec:GeneralProp} we prove some general properties of the isoperimetric profile function of a non-compact Riemannian manifold, in Section \ref{Sec:Existence} we prove the main theorems (we also give an alternative proof of Theorem \ref{thm:ExMinMain} in the case $k_0=0$ using the second variation or using differential inequalities) and we conclude in Section \ref{Sec:Examples} with a discussion of the examples of manifolds satisfying the assumptions of Theorem \ref{thm:ExMinMain}.

\subsection{Acknowledgment}
The project started when the first author was a Ph.D. student at SISSA and the second author visited SISSA thanks to the support of the M.U.R.S.T, within the project B-IDEAS ''Analysis and Beyond'' directed by Prof. Andrea Malchiodi. The first author acknowledges also the support of the ERC project ''GeMeThNES'' directed by Prof. Luigi Ambrosio; part of this work was written while the second author was a post-doctoral fellow at the University of S$\rm{\tilde{a}}$o Paolo supported by Fapesp grant 2010/15502-3, he also thanks IME-USP.

The authors want to express their deep gratitude to  Frank Morgan, Pierre Pansu,  Manuel Ritor\'e and Cesar Rosales for stimulating conversations at the first stages of this project. They also thank Michael Deutsch for reading the final manuscript.

\section{Notation and preliminaries}\label{Sec:NP}
Let $(M^n,g)$ be a smooth complete Riemannian $n$-manifold. The $n$-dimensional and $k$-dimensional Hausdorff measures of a set $\Omega\subset M$ will be denoted by $V(\Omega)$ and ${\cal H}^k(\Omega)$, respectively. For any measurable set $\Omega\subset M$ we denote with $\P(\Omega)$ the \emph{perimeter} of $\Omega$ defined by
$$\P(\Omega):=\sup \left\{\int_\Omega {\rm div} X\, d {\cal H}^{n+1}: |X|_{\infty}\leq 1 \right\},$$ 
where $X$ is a smooth vector field with compact support in $M$, $|X|_{\infty}$ is the sup-norm, and ${\rm div} X$ is the divergence of $X$.

A measurable subset $\Omega\subset M$ is of \emph{finite perimeter} if $\P(\Omega)< \infty$ and we denote with $\tau_M$ the family of all finite perimeter subsets of $M$. A  finite perimeter set $\Omega$ is said  \emph{indecomposable} if there do not exist disjoint non-empty  finite perimeter sets $\Omega_1,\Omega_2$ of positive volume such that $\Omega=\Omega_1\cup \Omega_2$,  and $\P(\Omega)=\P(\Omega_1)+ \P(\Omega_2)$ (for more details see \cite{ACMM}).

The \emph{isoperimetric profile} of $M$ is the function $I_M:(0,V(M))\to [0,+\infty)$ given by
   $$I_M(v):= \inf\{\P(\Omega): \Omega\in \tau_M, V(\Omega )=v \}.$$ 
If there exists a finite perimeter set $\Omega\in \tau_M$ satisfying $V(\Omega)=v$ and $I_M(v)=\P(\Omega)$, such an $\Omega$ will be called an \emph{isoperimetric region}, and we say that $I_M(v)$ is \emph{achieved}. A \emph{minimizing sequence} of sets of volume $v$ is a sequence of finite perimeter sets $\{\Omega_k\}_{k \in \N}$ such that $V(\Omega)=v$ for all $k \in \N$ and $\lim_{k \to \infty} \P(\Omega_k)=I_M(v)$. Recall that a sequence   $\{\Omega_k\}_{k \in \N}$ \emph{converges in the finite perimeter sense} to a set $\Omega$ if $\chi_{\Omega_k}\to \chi_{\Omega}$ in $L^1_{loc}(M)$ and $\lim_{k\to \infty} \P(\Omega_k)=\P(\Omega)$, where $\chi_{\Omega_k}$ and $\chi_{\Omega}$ denote the characteristic functions of $\Omega_k$ and $\Omega$, respectively.

Of course the existence of isoperimetric regions does not always occur in general, but if an isoperimetric region does exist, then the following classical regularity theorem holds (for the  proof see \cite{MorgReg}). 

\begin{pro}[Regularity]\label{prop:Reg}
Let $(M^n,g)$ be a smooth Riemannian $n$-manifold and  $v\in ]0,Vol(M)[$. Assume that the isoperimetric profile is achieved at $v$ by an open subset $\Omega \subset M$: $\P (\Omega)=I_M(v)$. Then
\begin{enumerate}
\item $\partial \Omega$ is the disjoint union of  a regular part $\partial \Omega_r$ and a singular one $\partial\Omega_s$. For each point $p\in \Omega_r$ there exists a neighborhood $U_p\subset M$ such that $\partial \Omega \cap U_p$ is a smooth hypersurface of constant mean curvature. Moreover the Hausdorff dimension of $\partial \Omega_s$ is less than or equal to $n-8$. In particular, if $n<8$ then $\partial \Omega_s=\emptyset$.   
\item $\partial \Omega$ is orientable and $\partial \Omega_r$ is equipped with a smooth outward pointing  unit normal vector field $\nu$.
\end{enumerate}  
\end{pro} 
This result was first obtained in the Euclidean setting by Gonzalez, Massari and Tamanini \cite{GMT} who treated interior regularity, and by Gr\"uter \cite{Gru}, who studied regularity near boundary points. Morgan \cite{Mor} generalized their results to the setting of Riemannian manifolds by using the paper of Almgren \cite{Alm}, which is Proposition \ref{prop:Reg}.
\begin{rem}\label{rem:reg}
In case the manifold $M^n$ and the metric $g$ are not smooth but regular enough there are still good regularity properties of isoperimetric regions. Indeed the standard interior Allard-type $C^{1,\alpha}$ regularity of (almost) minimizing boundaries away from a set of Hausdorff dimension at most 8 holds. This was shown by J. Taylor in \cite{Taylor} (this part of the discussion in her paper applies to $n$-dimensional manifolds). When the manifold is $C^2$ and the metric is Lipschitz, then this follows also from the work of R. Schoen and L. Simon \cite{SchSim} (for almost minimizing currents this was pointed out by B. White in \cite{Whi} pag. 498). When the manifold is $C^4$ and the metric $C^3$ so that the Nash embedding Theorem  provides an isometric embedding of $(M,g)$ into a high dimensional Euclidean space, then this also follows directly from  upon applying the Euclidean regularity theory as in \cite{Sim}. 
\end{rem}  

Now, in order to state the generalized existence theorem of the second author (a tool used throughout the paper), we recall the basics of the theory of   $C^{m,\alpha}$-pointed convergence of manifolds (for more details see \cite{Pet}).

\begin{df} \label{def:CmaConv}
Let $m \in \N$, $\alpha\in [0,1]$, $(M,g)$ be a $C^{m+1,\alpha}$-manifold with the $C^{m,\alpha}$-metric $g$ and let $p\in M$. A sequence of pointed smooth complete Riemannian $n$-manifolds is said to converge in
the pointed $C^{m,\alpha}$-topology to the manifold $(M,g,p)$, and we  write $(M_i, g_i, p_i)\rightarrow (M,g,p)$, if for every $R > 0$ we can find
a domain $\Omega_R$ with $B(p,R)\subseteq\Omega_R\subseteq M$, a natural number $\nu_R\in\mathbb{N}$, and $C^{m+1,\alpha}$-embeddings $F_{i,R}:\Omega_R\rightarrow M_i$ for large $i\geq\nu_R$ such that
$B(p_i,R)\subseteq F_{i,R} (\Omega_R)$ and $F_{i,R}^*(g_i)\rightarrow g$ on $\Omega_R$ in the $C^{m,\alpha}$ topology.
\end{df}

\begin{rem}
Whitney proved (see for instance Theorem 2.9 in \cite{Hirsch}) that if $\alpha$ is a $C^r$ differentiable structure on a topological manifold $M$, $r\geq 1$, then for every $r< s \leq \infty$ there exists a compatible $C^s$ differentiable structure $\beta\subset \alpha$, and $\beta$ is unique up to $C^s$ diffeomorphism. Therefore the assumption that $M$ is a $C^{m+1,\alpha}$-manifold is somehow unnecessary, but we will keep it for coherence with the literature.
\end{rem}
Now let us recall the notions of  bounded geometry and $C^{m,\alpha}$-bounded geometry.

\begin{df}\label{def:bddGeom}
A complete Riemannian $n$-manifold $(M,g)$ has \textbf{bounded geometry} if the following holds: 
\begin{enumerate}
\item There exists a constant $k\in \R$ such that $\Ric_g \geq k (n-1) g$, 
\item The volume of unit balls is uniformly bounded below: $\inf_{p \in M} V(B(p,1))\geq v_0>0$. 
\end{enumerate}
\end{df}

\begin{rem}
Notice that if $(M,g)$ has positive injectivity radius, $Inj_M>0$, then the second condition above is satisfied. Indeed Croke proved (see Proposition 14 in \cite{Croke1} and the discussion at page 2 in \cite{Croke2}; see also \cite{Berger} ) that there exists a  constant $C_n$ (depending only on $n=dim M$) such that if $r\leq \frac{Inj_M}{2}$ then $Vol(B(p,r))\geq C_n r^n$ for every $p \in M$.  
\end{rem}

\begin{df}\label{def:CmaBddGeom}
A complete Riemannian $n$-manifold $(M,g)$ has \textbf{$C^{m,\alpha}$-bounded geometry} if it has bounded geometry and moreover the following holds: 
For every diverging sequence of points $(p_j)_{j\in \N}$ there exists a subsequence $(p_{j_l})_{l\in \N}$ and a pointed $C^{m+1,\alpha}$-manifold $(M_\infty, g_\infty, p_\infty)$ with $C^{m,\alpha}$-metric  such that the sequence of pointed manifolds $(M,g,p_{j_l})\to (M_{\infty},g_{\infty},p_{\infty} )$ in $C^{m,\alpha}$ topology. 
\end{df}

 Now we recall the generalized existence theorem of the second author (Theorems 1 and  2 in \cite{NardAJM}). 
 \begin{thm}\label{thm:genEx}
 Let $(M,g)$ be a Riemannian $n$-manifold with $C^{1,\alpha}$-bounded geometry in the sense of Definition \ref{def:CmaBddGeom}. Then for every volume $v\in]0,V(M)[$ there are a finite number of limit manifolds  at infinity (precisely the manifolds at infinity are $C^{2,\alpha}$ with $C^{1,\alpha}$ metric) such that their disjoint union with $M$ contains an isoperimetric region of volume $v$ and perimeter $I_M(v)$.

 More precisely  for every volume $v\in]0,V(M)[$ there exist $N\in\mathbb{N}$, positive volumes $\{ v_i\}_{i\in\{1,...N\}}$, $N$ sequences of points $(p_{i,j})$,  $i\in\{1,...N\}$, $j\in\mathbb{N}$, $N$ limit manifolds $(M_{i,\infty}, g_{i, \infty}, p_{i, \infty})_{i\in\{1,...N\}}$ (precisely $M_{i,\infty}$ are $C^{2,\alpha}$-manifolds with $C^{1,\alpha}$-Riemannian metric $g_{i,\infty}$), and  finite perimeter sets $D_{i, \infty}\subseteq M_{i,\infty}$ such that  
 \begin{enumerate}
        \item  $\forall h\neq l$, $dist(p_{h,j},p_{l,j})\rightarrow +\infty$, as $j\rightarrow +\infty$,
        \item  $(M, g, p_{i,j})\rightarrow (M_{i,\infty}, g_{i, \infty}, p_{i, \infty})$ in $C^{1,\alpha}$ topology,
        \item   $v=\sum_{i=1}^N v_i$,
        \item  the volume in metric $g_{i, \infty}$ of $D_{i, \infty}$ equals $v_i$: $V_{g_{i, \infty}}(D_{i, \infty})=v_i$,
        \item the perimeter in metric $g_{i, \infty}$, $\P_{g_{i, \infty}}( D_{i, \infty})=I_{M_{i,\infty}}(v_i)$; that is $D_{i, \infty}$ is an isoperimetric region in $M_{i,\infty}$ for its own volume $v_i$,
        \item  $I_{M_{i,\infty}}(v_i)\geq I_M(v_i)$, 
        \item  $I_M(v)=\sum_{i=1}^N \P_{g_{i, \infty}}(D_{i, \infty})=\sum_{i=1}^N I_{M_{i,\infty}}(v_i)$,
        \item  the subset $D$ of the disjoint union $\mathring{\bigcup}_{i=1}^N  M_{i,\infty}$ defined as $D= \cup_{i=1}^N D_{i,\infty}$ is an isoperimetric region in volume $v$ in the manifold $\mathring{\bigcup}_{i=1}^N  M_{i,\infty}$.
 \end{enumerate}
\end{thm}

\begin{rem}\label{rem:GenEx}
The assumption about $C^{1,\alpha}$ bounded geometry was used  in the proof  of the previous theorem in \cite{NardAJM} to ensure that the  manifolds at infinity are at least $C^{2,\alpha}$ with $C^{1,\alpha}$ metric. Actually if one assumes a priori that the pointed $C^0$-limits are smooth Riemannian manifolds,  then the same generalized existence theorem holds. This is because the $C^0$-convergence of the metric tensors ensures the converge of the volume and of the perimeter (this is clear on smooth sets, so by approximation it holds on all finite perimeter sets).   
\end{rem}

Recall also the following useful result (see Theorem 3 in \cite{NardAJM}).
\begin{thm}\label{thm:IsopRegbdd}
Let $(M,g)$ be a complete Riemannian manifold with bounded geometry in the sense of Definition \ref{def:bddGeom}. Then isoperimetric regions are bounded.
\end{thm} 
Now we recall the notion of pseudo-bubble which will be useful to study the existence of isoperimetric regions of small volume.
Call $U_p M$ the fiber over $p$ of the unit tangent bundle (also called the sphere bundle) of the Riemannian manifold $(M,g)$. 

\begin{df}\label{def:pseudBub}
A \textbf{pseudo-bubble} is a hypersurface $\Psi B$ embedded in $M$ such that there exists a point $p\in M$ 
          and a function $w$ belonging to $C^{2,\alpha}(U_p M\backsimeq\mathbb{S}^{n-1},\mathbb{R})$, such that $\Psi B$ is the graph of
$w$ in normal polar coordinates centered at $p$, i.e. 
$$\Psi B= \left\lbrace exp_p (w(\theta )\theta ), \quad \theta \in U_p M \right \rbrace$$
 and  the mean curvature $H(w)=H_0+\phi$ of the normal graph is a real constant $H_0$ plus a function $\phi$, where $\phi$ is a first spherical harmonic function on $U_p M\backsimeq\mathbb{S}^{n-1}$. 
\end{df}
Recall also  the notion of \textit{Riemannian center of mass}.

\begin{df}\label{def:CenterMass}
Let $\Sigma\subset M$ be an embedded compact hypersurface in the $n$-dimensional Riemannian manifold $(M,g)$ and let $\mu$ the induced  volume  measure on $\Sigma$. Consider the function $\mathcal{E}_\Sigma: M\rightarrow [0,+\infty[$
$$\mathcal{E}_\Sigma (x):= \int_{\Sigma} d^2 (x, y)d\mu (y),$$
where $d$ is the Riemannian distance on $M$. The \textbf{center of mass} of $\Sigma$ is the  minimum point of $\mathcal{E}_{\Sigma}$ in $M$.
\end{df}\noindent
 Notice that, since $\Sigma$ is compact, by the Dominated Convergence Theorem, the function $\mathcal{E}_\Sigma$ is continuous and coercive, hence the existence of a minimum is guaranteed. Notice also that although uniqueness of this minimum point does not hold in general, it does in the cases we are interested, namely pseudo-bubbles of small diameter.

\section{Some general properties of the isoperimetric profile valid for (possibly non-compact) manifolds of bounded geometry} \label{Sec:GeneralProp}

Some classical properties of the isoperimetric profile for compact manifolds are also valid for non-compact manifolds (sometimes assuming  bounded geometry)   This section is devoted to prove some of them.

\begin{pro}\label{pro:ContI}
Let $(M,g)$ be a Riemannian manifold with $C^{2,\alpha}$-bounded geometry. Then the isoperimetric profile $I_M:]0,V(M)[\to [0,+\infty[$ is absolutely  continuous and twice differentiable almost everywhere. 
\end{pro}

\begin{pf} See Corollary 1 in \cite{NardAJM}.
\end{pf}

The following theorem is stated and proved in \cite{MorgJoh} (Theorem 3.4-3.5) in case of a compact ambient manifold but, as was pointed out to the authors by C. Rosales, the same proof holds for manifolds which are merely complete. 
\begin{pro}\label{prop:MJBishop}
Let $(M,g)$ be a smooth, complete, connected $n$-dimensional Riemannian
manifold and assume the following lower bound on the  Ricci curvature:
$$\Ric_{(M,g)}(.,.)\geq (n-1)k_0 g(.,.)\quad k_0\in \R.$$
 For a given volume $v \in ]0,V(M)[$, let $\Omega\subset M$ be an isoperimetric region of volume $v$ and perimeter $\P_M(\Omega)$. Then 
 $$\P_M(\Omega)\leq \P_0(B_v)$$ 
 where $\P_0(B_v)$ is the perimeter of a ball $B_v$ of volume $v$ in the simply connected space form $M_0$ of constant
curvature $k_0$. Suppose further that for some volume $v_0$, $I_M(v_0)=I_{M_0}(v_0)$. Then $M$ has constant sectional curvature $k_0$ and all metric balls of volume $v_0$ contained in $M$ are isometric to $B_{M_0}(v_0)$, the metric ball of volume $v_0$ contained in  $M_0$.
\end{pro}

\begin{pf} Fix a volume $v_0\in ]0,V(M)[$ and take any metric ball $B_M(v_0)\subset M$ of volume $v_0$. Since the proof of Theorem 3.5 in \cite{MorgJoh} (stated for compact $M$) relies only on the part of $M$ \emph{inside} the metric  ball, the same argument holds for complete (possibly non-compact) $M$. It follows that 
\begin{equation}\label{eq:AMAM0}
\P_M(B_M(v_0))\leq \P_0(B_{M_0}(v_0))
\end{equation}
where $B_{M_0}(v_0)$ is the ball of volume $v_0$ in the space form of constant sectional curvature $k_0$, moreover if equality holds then $B_M(v_0)$ is isometric to $B_{M_0}(v_0)$. 

Since $B_M(v_0)$ is a competitor in $M$ between regions of volume $v_0$ while $B_{M_0}(v_0)$ is minimizer in $M_0$, we have 
$$I_M(v_0)\leq \P_M(B_M(v_0))\leq \P_0(B_{M_0}(v_0))= I_{M_0}(v_0).$$
If $I_M(v_0)=I_{M_0}(v_0)$ then all metric balls of volume $v_0$ contained in $M$ are isometric to $B_{M_0}(v_0)$. Covering of $M$ with metric balls of volume $v_0$, we conclude that $M$ has constant sectional curvature $k_0$.  
\end{pf}

Now, using geometric differential inequalities, we are going to prove two useful  properties of the isoperimetric profile of a manifold with $C^{2,\alpha}$-bounded geometry. First recall that given a function $f:\R \to\R$, one denotes 
\begin{equation}\label{def:D2f}
\overline{D^2 f}(x_0) := \limsup_{h\to0^+} \frac{f (x_0+h)+f (x_0-h)-2f (x_0)}{h^2}, \quad \forall x_0\in \R.
\end{equation}

The following theorem for compact manifolds is due to Bayle (see \cite{BayTh} Theorem 2.2.1).

\begin{thm}\label{thm:DiffIneRic}
Let $(M^n,g)$ be a complete $n$-dimensional Riemannian manifold,  of $C^{2,\alpha}$-bounded geometry with $n\geq 2$. Let us assume that 
$$\Ric_g\geq (n-1)k_0\, g, \quad k_0 \in \R .$$
Then the normalized isoperimetric profile $Y_{(M,g)}:=I_M^{\frac{n}{n-1}}$ satisfies the following second order differential inequality
\begin{equation}\label{eq:diffIne}
\forall v>0\quad \overline{D^2Y_{(M,g)}} (v)\leq -n k_0\; Y_{(M,g)}(v)^{\frac{2-n}{n}},
\end{equation} 
with equality in the case of the simply connected space form of constant sectional curvature $k_0$; moreover if
equality holds for a certain $v_0$, then  all the isoperimetric regions with volume $v_0$ have totally umbilic boundary along which the Ricci curvature, evaluated on unit normal directions, equals $(n-1)k_0$. 

\end{thm}

\begin{pf} The proof of Bayle relies on the Ricci curvature lower bound and on the fact that for every $v\in ]0, V(M)[$ the isoperimetric profile $I_M(v)$ is achieved by a region with the regularity stated in Proposition \ref{prop:Reg}. Using the Generalized Existence Theorem \ref{thm:genEx}, for every $v>0$ there exists an isoperimetric region $D=D_1\cup_{i=2}^N D_{\infty,i}$ where $D_1\subset M, D_{\infty,i}\subset M_{\infty,i}$ are isoperimetric regions in their own manifolds; moreover $D$ is an isoperimetric region in the manifold given by the disjoint union $M \mathring{\cup}_{i=2}^\infty M_{\infty,i}$ for its own volume. Therefore, recalling Remark \ref{rem:reg} and observing that the lower bound on the Ricci curvature passes to the $C^{2,\alpha}$-limit manifolds $(M_{\infty,i}, g_{\infty,i})$,  the argument of Bayle  (see \cite{BayTh} Theorem 2.2.1 and the computations on pages 45-51; see also \cite{MorgJoh} Proposition 3.3) can be repeated, bringing the desired conclusion.
\end{pf}

\begin{cor}\label{cor:RicPos}
Let $(M,g)$ be complete, non-negatively Ricci curved, $\Ric_g \geq 0$, and of $C^{2,\alpha}$-bounded geometry. Then the isoperimetric profile function $I_M:[0,V(M)[\to \R$ is strictly concave.  In particular, since $I_M(0)=0$, $I_M$ is strictly subadditive, and this implies that every isoperimetric region is indecomposable. 
\end{cor}

\begin{pf}
Since in this case $k_0=0$, by the differential inequality \eqref{eq:diffIne} we get
$$\forall v>0\quad \overline{D^2Y_{(M,g)}}(v) \leq 0,$$
so the function $Y_M$ is concave (see \cite{BayTh} Proposition B.2.1 pag.181). Now observe that $I_M=Y_{(M,g)}^{\frac{n-1}{n}}$; since the exponent is $\frac{n-1}{n}<1$, it follows that $I_M$ is strictly concave. Of course a continuous strictly concave function on $]0,\infty[$ which is null at $0$ is strictly subadditive (for the simple proof see for example \cite{BayTh}, Lemma B.1.4). Now let  $\Omega_v$ be an isoperimetric region in volume $v>0$. If by contradiction $\Omega_v=\Omega^1 \cup \Omega^2$ is a decomposition of $\Omega$, say $0<v_1=V(\Omega_1)$ and $0<v_2=V(\Omega_2)$, then by  the  assumed subadditivity of the isoperimetric profile we reach the contradiction
$$I_M(v)=\P(\Omega_v)=  \P(\Omega_1)+\P(\Omega_2)\geq I_M(v_1)+I_M(v_2)>I_M(v).$$ 
\end{pf}

\section{Existence and properties of isoperimetric regions}\label{Sec:Existence}

\subsection {Proof of Theorem \ref{thm:ExMinMain}}

Recall that the isoperimetric regions in the $n$-dimensional simply connected space form $\mathbb{M}_{{k}_0}^n$ of constant sectional curvature $k_0$ are metric balls (no matter where the center is). Therefore it is clear that, applying the generalized existence Theorem \ref{thm:genEx} (recall also Remark \ref{rem:GenEx}) to a manifold $(M^n,g)$ satisfying the assumptions of Theorem \ref{thm:ExMinMain}, there is at most one component of the generalized isoperimetric region $D$ placed in the manifold $\mathbb{M}_{{k}_0}^n$ at infinity. More precisely, fixing a positive volume $v>0$ and considering $D$ as a generalized isoperimetric region for the volume $v$ given in $8$ of Theorem \ref{thm:genEx}, we have that  
\be\label{eq:main1}
D=D_1 \cup D_{\infty} \quad \text{ where } D_1 \subset M \text{ and } D_{\infty}\subset \mathbb{M}_{{k}_0}^n,
\ee
with $v=v_1+v_\infty$ where $v_1=V_M(D_1)$ and $v_\infty:=V_{\mathbb{M}_{{k}_0}^n}(D_\infty)$.
Since both  $D_{\infty}\subset \mathbb{M}_{{k}_0}^n$ and $D_1\subset M$ are isoperimetric regions for their own volume, we have that 
\begin{eqnarray}
 D_{\infty}&\subset& \mathbb{M}_{{k}_0}^n\quad  \text {is a metric ball: } D_\infty=B_{\mathbb{M}_{{k}_0}^n}(v_\infty), \label{eq:Dinfty} \\
 D_1&\subset& M \quad \;\text{ is bounded} \label{eq:D_1},
\end{eqnarray}
where $B_{\mathbb{M}_{{k}_0}^n}(v_\infty)$ is a metric ball in $\mathbb{M}_{{k}_0}^n$ of volume $v_\infty$ and  the second statement is ensured by Theorem \ref{thm:IsopRegbdd}.

If  $D_\infty = \emptyset$ the conclusion follows, so we can assume that $D_\infty \neq \emptyset$ and $v_\infty:=V_{\mathbb{M}_{{k}_0}^n}(D_\infty)>0$. 
Let us consider a metric ball  $B_M(v_\infty)\subset M$ of volume $v_\infty$ placed at positive distance from $D_1$ (this is possible thanks to \eqref{eq:D_1} and the assumed asymptotic behaviour of $(M,g)$). By formula \eqref{eq:AMAM0} in the proof of Proposition \ref{prop:MJBishop}, we have that 
\be \label{eq:BMDinf}
\P_M(B_M(v_\infty))\leq \P_{\mathbb{M}_{{k}_0}^n}(D_\infty).
\ee
Therefore if we move all the volume $v_\infty$ which stays in the manifold at infinity $\mathbb{M}_{{k}_0}^n$  in any metric ball contained in the original manifold  $M$, we do not increase the perimeter and
$$I_M(v)=\P_M(D_1)+ \P_{\mathbb{M}_{{k}_0}^n}(D_\infty)\geq \P_M(D_1)+ \P_M(B_M(v_\infty))=\P_M( D_1 \cup  B_M(v_\infty)),$$
where we used 7 of Theorem \ref{thm:genEx} for the first equality and the fact that $D_1$ and $B_M(v_\infty)$ are at positive distance for the final equality.

Since $D_1$ and $B_M(v_\infty)$ are disjoint, then $V(D_1 \cup  B_M(v_\infty))=v_1+v_\infty=v$ and we conclude that $D_1\cup B_M(v_\infty)$ is an isoperimetric region in $M$ for the volume $v$. Since $v>0$ was arbitrary the theorem is proved.

The indecomposability of the isoperimetric regions in case $\Ric_g\geq 0$ is ensured by Corollary \ref{cor:RicPos}.
\hfill$\Box$ 

\subsection{The case $\Ric_g\geq 0$}

Since the asymptotically  locally euclidean manifolds with non-negative Ricci tensor are particularly interesting for the applications (see Section \ref{Sec:Examples}), in this subsection we give an alternative proof of Theorem \ref{thm:ExMinMain} in this case. Moreover, if $\Ric_g\geq 0$, it is possible to do a finer analysis on the minimizing sequences: roughly speaking  either they converge to an isoperimetric region or they diverge, but they cannot split into a converging part and a diverging part. 

\subsubsection{An alternative  proof of Theorem \ref{thm:ExMinMain} in the case $\Ric_g\geq 0$} \label{SS:thmExMin}
Let $(M^n,g)$ satisfy the hypothesis of Theorem \ref{thm:ExMinMain} with $k_0=0$, so that $\Ric_g\geq 0$ and the manifold is $C^0$-locally asymptotic to $\Rn$. For a fixed $v>0$,  we want to show that there exists an isoperimetric region in $M$ of volume $v$.  Theorem \ref{thm:genEx} (recall also Remark \ref{rem:GenEx}) ensures the existence of a generalized isoperimetric region 
$$D=D_1 \cup D_{\infty} \quad \text{where } D_1 \subset M \text{ and } D_{\infty}\subset \Rn$$
are such that $V_M(D_1)=v_1$, $V_{\Rn}(D_\infty)=v_\infty$ with $v_1+v_\infty=v$ and $D_1$ (resp. $D_\infty$) is an isoperimetric region in $M$ (resp. in $\Rn$) for its own volume $v_1$ (resp. $v_\infty$).

The structure of the proof is the following: first we  show that $D$ is connected, so either $D=D_1$ or $D=D_{\infty}$, then we prove that it must be $D=D_1$. \\

\emph{STEP 1}: $D=D_1\subset M$ or $D=D_\infty\subset \Rn$.
\\Let us start assuming $dim(M)=n<8$, since in this case the proof is very short (later we will explain how to handle the general case). As $D$ is an isoperimetric domain in $M\cup \Rn$, its boundary is a smooth  stable CMC hypersurface of finite area. If by contradiction $D_1\neq \emptyset$ and $D_\infty \neq \emptyset$, then  $0<\P_M(D_1),  \P_{\Rn}(D_\infty) <\infty$ and there exist $c_1,c_\infty\in \R\backslash\{0\}$ such that
\begin{equation}\label{eq:c1c2}
c_1 \,\P_M(D_1)=c_\infty\, \P_{\Rn}(D_\infty). 
\end{equation}
Denote by $\nu_1$ and $\nu_{\infty}$ the outward pointing unit normal vectors to $\partial D_1$ and $\partial D_\infty$, and consider the variation of $D$ composed by varying $D_1$ in the direction $c_1 \nu_1$ and varying $D_\infty$ in the direction of $-c_\infty \nu_\infty$. Observe that \eqref{eq:c1c2} implies that this is an admissible variation (it has null mean value so it is volume preserving to first order). 
Since the first variation of the perimeter $\P$ of  $D$ with respect to null mean value deformations is null (recall that $\partial D$ is union of smooth hypersurfaces of constant mean curvature), it is interesting to compute the second variation of $\P$ in the specified direction. The standard expression of the second variation of the area  (see for example \cite{BDE}, Proposition 2.5) gives
\begin{equation}\label{eq:SVF}
\P''(D)=\P_M''( D_1)+\P_{\Rn}''(D_\infty)=-c_1^2\int_{\partial D_1} \left(\sigma_{1}^2+ \Ric_g(\nu_1,\nu_1)\right)-c_\infty^2 \int_{\partial D_\infty} \sigma_\infty^2, 
\end{equation}
where $\sigma_1$ (resp. $\sigma_\infty$) is the norm of the second fundamental form of $\partial D_1$ (resp. $\partial D_\infty$). Now observe that $\partial D_\infty$ is an $(n-1)$-dimensional Euclidean sphere of radius $r_\infty$, so 
$$\int_{\partial D_\infty} \sigma_\infty^2= (n-1) \omega_{n-1} r_\infty^{n-3}<0$$
where $\omega_{n-1}$ is the perimeter of unit sphere in $\Rn$. Since we are assuming that $\Ric_g\geq 0$, we can conclude that
\begin{equation}\label{eq:contr}
\P''(\partial D)\leq  -c_\infty^2 (n-1) \omega_{n-1} r_\infty^{n-3}<0, 
\end{equation}
which contradicts the stability of $\partial D$.

In the general case of dimension $n$ for the ambient manifold $M$, we can use a classical argument employing cutoff functions. This trick was attributed in \cite{BaPa} (Section 7)  to P. Berard, G. Besson, S. Gallot, proved in detail in \cite{BayTh} (Proposition A.0.5) and in \cite{MorRit}  (Lemma 3.1), and used for example in \cite{MorgJoh} in Section 2. The argument is as follows: If $\Omega\subset M$ is an isoperimetric region and with $\partial\Omega_r, \partial\Omega_s$ the regular and the singular part respectively, then by the aforementioned Proposition A.0.5 in \cite{BayTh} there exist a function $\phi_\epsilon:\partial\Omega\rightarrow [0,1]$ for each $\epsilon>0$ with the following properties:
\begin{itemize}
\item ${\phi_\epsilon}_{|\partial\Omega_r}\in C^{\infty}(\partial\Omega_r, [0,1])$, with $\spt(\phi_\epsilon)\subset\subset\partial\Omega_r$;
\item $\P(\Omega)-\epsilon\leq\int_{\partial\Omega} \phi_\epsilon d\mathcal{H}_{n-1}\leq \P(\Omega)$,
\item $\P(\Omega)-\epsilon\leq\int_{\partial\Omega} \phi_\epsilon^2 d\mathcal{H}_{n-1}\leq \P(\Omega)$,
\item $\frac{\int_{\partial\Omega}||\nabla\phi_\epsilon||^2d\mathcal{H}_{n-1}}{\int_{\partial\Omega}\phi_\epsilon^2d\mathcal{H}_{n-1}}\leq\epsilon$.
\end{itemize}
For small $\epsilon>0$, consider the standard expression of $\P''(D)$ with variation field $c_{1,\epsilon}\phi_\epsilon\nu_1-c_{\infty,\epsilon}\nu_\infty$, where $\phi_\epsilon$ is as before with $\Omega = D_1$ and $c_{1,\epsilon},c_{\infty,\epsilon}$ chosen in such a way that the variation has null mean value.  Letting $\epsilon\rightarrow 0$ in this expression gives again \eqref{eq:contr}, which contradicts the stability. 
\newline

\emph{An alternative proof of STEP 1}:
\\Since the enlarged manifold $M\cup \Rn$ has non-negative Ricci curvature and $C^0$-bounded geometry, by Corollary \ref{cor:RicPos} (notice that we asked $C^{2,\alpha}$-bounded geometry just to ensure that the manifolds at infinity were smooth enough to carry the regularity of isoperimetric regions and for ensuring that the lower bound on the Ricci tensor is preserved in the limit; both facts are clearly true if the limit manifolds are isometric to the Euclidean $n$-dimensional space) the isoperimetric regions are indecomposable, so either  $D=D_1 \subset  M$ or $D=D_\infty \subset \Rn $.
\newline 

\emph{STEP 2}: $D=D_1\subset M$.
\\By Step 1, either $D=D_1\subset  M$ or $D=D_\infty\subset \Rn$.  If $D=D_1$ we have finished, so we can assume $D=D_\infty$. Recalling that $V(D)=v$, items 5 and 7 of Theorem \ref{thm:genEx} yield that
\be \label{eq:IMIRn}
I_M(v)=\P_{M \cup \Rn}(D)=\P_{\Rn}(D_\infty)=I_{\Rn}(v).
\ee
Now, by  Proposition \ref{prop:MJBishop} all the metric balls $B_M(p_0,v)\subset M$ of volume $v$ are isometric to the standard round ball $B_{\Rn}(v)\subset \Rn$ of volume $v$, and in particular the area of the boundaries are equal.  We conclude
$$I_M(v)\leq P_M(B_M(p_0,v))=\P_{\Rn}(B_{\Rn}(v))=I_{\Rn}(v)=I_M(v)$$
where we used \eqref{eq:IMIRn} in the last equality. Therefore $I_M(v)= \P_M(B_M(p_0,v))$, i.e. $B_M(p_0,v)$ is an isoperimetric region in volume $v$ for every $p_0\in M$, and the theorem follows by the arbitrarity of $v>0$. 
\\We remark that in the latter case $M$ is locally isometric to $\Rn$.

\hfill$\Box$

\subsubsection{Finer properties of the minimizing sequences in the case $\Ric\geq 0$}
Given a Riemannian manifold $(M,g)$, we say that a sequence $\{\Omega_k\}_{k\in \N}$ of finite perimeter subsets of $M$ \emph{diverges} if for every compact subset $K\subset \subset M$, there exists $N \in \N$ such that 
$$\Omega_k \cap K= \emptyset \quad \forall k \geq N.$$
Let us recall the following fundamental theorem of Ritor\'e and Rosales (Theorem 2.1 in \cite{RitRosal})

\begin{thm}[Ritor\'e-Rosales '04]\label{thm:RR}
Let $(M^n,g)$ be a complete connected Riemannian $n$-manifold. For every minimizing sequence $\{\Omega_k\}_{k \in \N}$ of sets of volume $v$, there exists a finite perimeter set $\Omega\subset M$ and sequences of sets of finite perimeter $\{\Omega^c_k\}_{k \in \N}$, $\{\Omega^d_k\}_{k \in \N}$, with $\Omega_k=\Omega^c_k \cup \Omega^d_k$ and $\Omega^c_k \cap \Omega^d_k= \emptyset$, such that the following hold:
\begin{enumerate}
\item $V(\Omega)\leq v$, $\P(\Omega)\leq I_M(v)$,
\item $V(\Omega^c_k)+V(\Omega^d_k)=v$, $\lim_{k\to \infty}[\P(\Omega^c_k)+\P(\Omega^d_k)]=I_M(v)$,
\item the sequence $\{\Omega^d_k\}_{k\in \N}$ diverges,
\item there exists a finite perimeter set $\Omega\subset M$ such that, passing to a subsequence $\{k_j\}_{j \in \N}$, $\{\Omega^c_{k_j}\}_{j \in \N}$ converges to $\Omega$ in the sense of finite perimeter sets. In particular, $\lim_{j \to \infty} \P(\Omega^c_{k_j})=\P(\Omega)$ and $\lim_{k_j\to \infty} V(\Omega^c_{k'})=V(\Omega)$,
\item $\Omega$ is an isoperimetric region (possibly empty) for the volume it encloses. 
\end{enumerate}
\end{thm}

The aim of the present section is to prove the following theorem, which says that if $\Ric_g\geq 0$ then for every $v>0$ any minimizing sequence $\{\Omega_k\}_{k\in \N}$ for the volume $v$ cannot split into a convergent part $\{\Omega^c_k\}_{k\in \N}$ and a divergent part $\{\Omega^d_k\}_{k \in \N}$ such that $\liminf_k V(\Omega^c_k)>0$ and $\liminf_k V(\Omega^d_k)>0$, or in other words, any minimizing sequence for the volume $v>0$ either converges to an isoperimetric region of volume $v$ or diverges, up to a part whose volume converges to zero and up to subsequences. 

\begin{thm}\label{thm:Ric>0}
Let $(M^n,g)$ be a complete connected Riemannian $n$-manifold having $C^{2,\alpha}$-bounded geometry and satisfying $\Ric_g\geq 0$. Fix $v\in ]0,V(M)[$ and consider any minimizing sequence $\{\Omega_k\}_{k\in \N}$ for the volume $v$. Then there exist sequences of sets of finite perimeter $\{\Omega^1_k\}_{k \in \N}$, $\{\Omega^2_k\}_{k \in \N}$ and a subsequence $\{k_j\}_{j\in \N}$ such that
$$\Omega_k=\Omega^1_k \cup \Omega^2_k, \;\Omega^1_k \cap \Omega^2_k= \emptyset\;\text{for all }k \in \N \text{, and } \quad  \lim_{j\to \infty} V(\Omega^1_{k_j})=v,\; \lim_{j\to \infty} V(\Omega^2_{k_j})=0.$$
Moreover, either $\{\Omega^1_{k_j}\}_{j\in \N}$ diverges or there exists an isoperimetric region $\Omega\subset M$ for the volume $v$ such that $\{\Omega^1_{k_j}\}_{j \in \N}$ converges to $\Omega$ in the sense of finite perimeter.
\end{thm} 

\begin{proof}
Applying Theorem \ref{thm:RR} to the minimizing sequence $\{\Omega_k\}_{k \in \N}$ we obtain the sequences of sets of finite perimeter $\{\Omega^c_k\}_{k \in \N}$, $\{\Omega^d_k\}_{k \in \N}$ with the stated properties. Let 
$$v_1=V(\Omega)=\lim_{j \to \infty} V(\Omega^c_{k_j}) \quad \text{and} \quad v_{\infty}=v-v_1=\lim_{j \to \infty} V(\Omega^d_{k_j}).$$
The conclusion follows if we prove that either $v_1=v$ and $v_\infty=0$ or that $v_1=0$ and $v_\infty=v$. 

Assume by contradiction that  both $v_1$ and $v_\infty$ are strictly positive. Combining items 2, 4 and 5 of Theorem \ref{thm:RR}, we infer
\be \nonumber
I_M(v)=\lim_{k\to \infty} [\P(\Omega^c_k)+\P(\Omega^d_k)]=I_M(v_1)+ \lim_{k\to \infty} \P(\Omega^d_k).
\ee
Using the trivial inequality $\P(\Omega^d_k)\geq I_M(V(\Omega^d_k))$ we can continue the chain above, obtaining
\be\nonumber 
I_M(v)\geq I_M(v_1)+\limsup_{k\to \infty} I_M(V(\Omega^d_k)) \geq I_M(v_1)+\lim_{j\to \infty} I_M(V(\Omega^d_{k_j}))=I_M(v_1)+I_M(v_\infty),
\ee
where, in the last equality, we used that $\lim_{j\to \infty} V(\Omega^d_{k_j})=v_\infty$ together with the continuity of the isoperimetric profile ensured by Proposition \ref{pro:ContI}. Therefore $I_M(v)\geq I_M(v_1)+I_M(v_\infty)$, and if both $v_1$ and $v_\infty$ are strictly positive, this contradicts the strict subadditivity of $I_M$ stated in Corollary \ref{cor:RicPos}.
\end{proof}

\subsection{Existence of isoperimetric regions of small volume under assumptions on the scalar curvature}

In this section we prove  Theorem \ref{thm:ExSmallVol}, the existence of isoperimetric regions of small volumes in non-compact manifolds of any dimension under assumptions on the scalar curvature alone.
\newline

PROOF OF THEOREM \ref{thm:ExSmallVol}: From Lemma 3.6 in \cite{NardNC}, there exists a small $v_0>0$ such that for any $0<v<v_0$, the isoperimetric profile $I_M(v)$ is achieved in the enlarged manifold $M\cup M_\infty$, where $M_\infty$ is given by a compactness argument in the theory of pointed convergence of manifolds (see \cite{NardNC}, and note that we have changed the notation a bit from that in the cited paper, where $M_\infty$ may coincide with $M$, while here $M$ denotes the original manifold and $M_\infty$ denotes the manifold we are attaching at infinity in case a minimizing sequence is diverging). From Lemma 3.7, the minimizer is a pseudo-bubble (for the precise notion see Definition \ref{def:pseudBub}) $\Psi B_v$ contained either in $M$ or in $M_\infty$.\\  

We now show that $\Psi B_v$ must be contained in $M$, from which the theorem follows. 
Suppose for contradiction that $\Psi B_v\subset M_\infty$. Then the expansion of the isoperimetric profile $I_{M_\infty}$ for small volume (see formula (2) in Theorem 2 in  \cite{NardNC}) is
\begin{equation}\label{eq:IMIMinfty}
I_M(v)=I_{M_\infty}(v)=c_n v^{\frac{n-1}{n}}\left(1-\frac{S_\infty}{2n(n+2)}\left(\frac{v}{\omega_n}\right)^\frac{2}{n}  + o\left( v^{\frac{2}{n}}\right)   \right), 
\end{equation}
where $c_n$ is the Euclidean isoperimetric constant, $S_\infty:=\sup_{M_\infty}\Scal_{g_\infty}$, and $\omega_n$ is the volume of the $n$ dimensional ball of radius $1$. Notice that since $(M,g)$ has $C^{2,\alpha}$-bounded geometry, the asymptotic bounds on the curvature of $M$ are transferred to the $C^{2,\alpha}$-limit manifold $M_\infty$, so under our assumptions we have that
\be\nonumber
S_\infty \leq S.
\ee

On the other hand, taking a point $\bar{p}\in M$ where $\Scal_g(\bar{p})>S$, the same computations show that on small geodesic balls $B_{\bar{p},v}$ of volume $v$ centered at $\bar{p}$ we have
\begin{equation}\label{eq:expAM}
\P_M(B_{\bar{p},v})=c_n v^{\frac{n-1}{n}}\left(1-\frac{\Scal_g(\bar{p})}{2n(n+2)}\left(\frac{v}{\omega_n}\right)^\frac{2}{n}  + o\left( v^{\frac{2}{n}}\right)   \right).
\end{equation}
Since $\Scal_g(\bar{p})>S\geq S_\infty$, the combination of \eqref{eq:IMIMinfty} and \eqref{eq:expAM} gives the contradiction 
$$I_M(v)\leq A_M(S_{\bar{p},v})<I_{M_\infty}(v)=I_M(v), \quad \text{for small }v>0.$$
Finally, from Theorem 1 in \cite{NardNC}, the isoperimetric regions of small fixed volume $v$ are pseudo-bubbles with center of mass $\bar{p}_v$ converging in Hausdorff distance to the set of points of global maximum of the scalar curvature $\Scal_g$ as $v \to 0$.

\section{Noteworthy examples of manifolds satisfying the assumptions of Theorem \ref{thm:ExMinMain} and open problems} \label{Sec:Examples}
\subsection{ALE Gravitational Instantons}

A first class of manifolds satisfying the assumptions of Theorem \ref{thm:ExMinMain} is given by the so-called Asymptotically Locally Euclidean (ALE) Gravitational Instantons: 4-manifolds which are solutions of the Einstein vacuum equations with null cosmological constant (i.e. they are Ricci flat, $\Ric \equiv 0$).  These are non-compact with just one end which is topologically a quotient of $\R^4$ by a finite subgroup of $O(4)$, and the Riemannian metric $g$ on this end is asymptotic to the euclidean metric up to $O(r^{-4})$,
$$g_{ij}=\delta_{ij}+O(r^{-4}),$$
with appropriate decay in the derivatives of $g_{ij}$ (in particular, these metrics are  $C^0$ locally asymptotic, in the sense of Definition \ref{def:CmaConv}, to the Euclidean $4$-dimensional space). 
\\

The first example of such manifolds was discovered by Eguchi and Hanson in  \cite{EH78}; the authors, inspired by the discovery of self-dual instantons in Yang-Mills Theory, found a self-dual ALE instanton metric.  The Eguchi-Hanson example was then generalized by Gibbons and Hawking \cite{GH} who constructed for each integer $k\geq 2$ a family of ALE 4-dimensional gravitational instantons depending on $3k-6$ parameters, which have self dual curvature and are asymptotic to a quotient of  $\R^4$ by a cyclic group of order $k$; these ''multi-Eguchi-Hanson'' metrics constitute the  building blocks of Euclidean quantum gravity theory (see \cite{H1}, \cite{H2}) and were obtained also by Hitchin \cite{Hi1}, who derived them through an application of Penrose's non-linear graviton construction. The ALE Gravitational Instantons were classified in 1989 by Kronheimer (see \cite{Kro1}, \cite{Kro2}).
\\

For the reader's convenience, in order to give at least one explicit example, we briefly describe the Eguchi-Hanson metric following \cite{EH79}.  Let $ds^2=dt^2+dx^2+dy^2+dz^2$ the Euclidean metric in $\R^4$ and observe that the flat metric can be written in polar coordinates as
\begin{equation}\label{eq:EucPolar}
ds^2=dr^2+r^2(\sigma_x^2+\sigma_y^2+\sigma_z^2)
\end{equation}
where $r^2=t^2+x^2+y^2+z^2$ and 
\begin{eqnarray}
\sigma_x&=&\frac{1}{r^2}(x dt-t dx+y dz-z dy) \nonumber \\
\sigma_y&=&\frac{1}{r^2}(y dt-t dy+z dx-x dz) \nonumber \\
\sigma_z&=&\frac{1}{r^2}(z dt-t dz+x dy-y dx). \nonumber 
\end{eqnarray}
Then the Eguchi-Hanson metric can be written
\begin{equation}\label{eq:EgHan}
ds^2=\left[1-\left(\frac{a}{r} \right)^4 \right]^{-1} dr^2+r^2(\sigma_x^2+\sigma_y^2)+r^2 \left[1-\left(\frac{a}{r} \right)^4 \right]  \sigma_z^2,
\end{equation}
where $a$ is a real constant. The metric is singular at $r=a$ in $\R^4$, but this singularity disappears if one identifies $(t,x,y,z)\sim (-t,-x,-y,-z)$, after which we obtain a smooth, geodesically complete, Ricci flat metric on $\R^4/ \sim$. The global topology of the manifold is the following: near $r=a$ the manifold has the topology of $\R^2\times S^2$ (more precisely, at every point of $S^2$ there is an $\R^2$ attached which shrinks to a point as $r\to a$) while for large $r$ the metric approaches the flat metric.  Notice that because of the identification $\sim$, the boundary at infinity is not $S^3$ but $\R P^3\cong SO(3)$, for which $S^2\cong SU(2)$ is the double cover. So, as remarked before, the manifold is locally asymptotically Euclidean, but the global topology at infinity differs from the that of $\R^4$. For completeness let us also recall that the entire manifold just described can be seen also as the cotangent bundle of the complex plane $\C P^1\cong S^2$.

{\bf{Open Problem 1}}: By the direct application of Theorem \ref{thm:ExMinMain} to the Eguchi-Hanson space, we get existence of isoperimetric regions for \emph{every} value of the area.  It is an interesting open problem to characterize such regions. Since the metric is radially symmetric, we expect that (at least for large volumes) the isoperimetric regions are the 3-dimensional projective spaces described by $\{r= const \}$.  

{\bf{Open Problem 2}}: Clearly  Theorem \ref{thm:ExMinMain} can be applied as well to the other more general ALE Gravitational Instantons mentioned above; the description of the isoperimetric regions is again an interesting open problem. We expect that for large volumes they are normal graphs of the quotient of large centered spheres.
\\

We remark that the existence and description of isoperimetric regions is an important issue in general relativity. To name a few examples, D. Christodoulou and S.T. Yau  proved in \cite{ChrYau} that the Hawking mass of isoperimetric spheres is non-negative (provided the scalar curvature of the ambient manifold is non-negative);  H. Bray in  \cite{BrayTh} gave a proof of a special case of the Riemannian Penrose inequality using isoperimetric techniques;  G. Huisken in \cite{HuRMFO} proposed a definition of mass using just isoperimetric concepts; H. Bray and F. Morgan  in \cite{BrMo} characterized isoperimetric regions in certain spherically symmetric  manifolds, in particular in Schwarzshild; M. Eichmair and J. Metzger in \cite{MeEi1}, \cite{MeEi2} and \cite{MeEi3} described the isoperimetric regions of large volume in initial data sets for the Einstein's equations;  J. Corvino, A. Gerek, M. Greenberg, B. Krummel in \cite{CGGK} characterized the isoperimetric regions in the spatial Reissner-Nordstrom and Schwarzschild anti-de Sitter manifolds; S. Brendle and M. Eichmair in \cite{BrEi} characterized isoperimetric regions in the ''doubled'' Schwarzschild manifold.

\subsection{Asymptotically Hyperbolic Einstein manifolds}
In this subsection we discuss the importance and existence of Einstein manifolds which are locally $C^0$-asymptotic to  a negatively curved space form (and hence satisfy the assumption of Theorem \ref{thm:ExMinMain}). 

Let $M$ be the interior of a compact $n$-dimensional manifold $\bar{M}$ with non-empty boundary $\partial M$; a complete metric $g$ on $M$ is $C^{m,\alpha}$ conformally compact if there is a defining function $\bar{\rho}$ on $\bar{M}$ such that the conformally equivalent metric $$\tilde{g}=\rho^2 g$$
extends to a $C^{m,\alpha}$ metric on the compactification $\bar{M}$. A defining function $\rho$ is a smooth, non-negative function on $\bar{M}$ with $\rho^{-1}(0)=\partial M$ and $d \rho\neq 0$ on $\partial M$. The induced metric $\gamma=\tilde{g}_{|\partial M}$ is called the  boundary metric associated to the compactification $\tilde{g}$. There are many possible defining functions and hence many compactifications of a metric $g$, so only the conformal class $[\gamma]$ of $\gamma$ on $\partial M$ is uniquely determined by $(M,g)$. If the metric $g$ is $C^2$ conformally compact and Einstein normalized so that
$$ \Ric_g= -(n-1) g,$$
then it is asymptotically hyperbolic in the sense that $|K_g+1|=O(\rho^2)$, where $K_g$ is the sectional curvature of $g$ (see for example the Appendix in \cite{AndGAFA}). The relationship with the hyperbolic space can be made even more explicit by constructing special coordinate charts near the boundary (see for example Chapter 3 in \cite{Lee}).
\\

In recent years, interest in asymptotically hyperbolic Einstein metrics has risen dramatically, also thanks to their physical relevance. Indeed the previous described notion of conformal infinity for a (pseudo)-Riemannian manifold was introduced by  Penrose \cite{Penrose} in order to analyze the behaviour of gravitational energy in asymptotically flat space times. More recently, asymptotically hyperbolic Einstein metrics have begun to play a central role in the ''AdS/CFT correspondence'' of quantum field theory: broadly speaking the correspondence states the existence of a duality equivalence between gravitational theories (such as string theory or $M$ theory) on $M$ and conformal field theories on the boundary at conformal infinity $\partial M$ (see for example \cite{And.Corr}).  
\\

Regarding the existence of such metrics, Graham and Lee \cite{GL} have proved that any metric $\gamma$ near the standard metric $\gamma_0$ on $S^{n-1}$ in a sufficiently smooth topology may be filled with an asymptotically hyperbolic Einstein metric  $g$ on the $n$-ball $B^n$ having prescribed boundary metric $\gamma$, and moreover such metrics have a conformal compactification with a certain degree of smoothness. More precisely, they prove that for any $m\geq 2$, there is an open neighborhood $U_{\gamma_0}$ of $\gamma_0$ in the space of $C^{m,\alpha}$ metrics on $S^{n-1}$ such that any metric $\gamma \in U_{\gamma_0}$ is the boundary metric of an asymptotically hyperbolic Einstein metric $g$ on the $n$-ball $B^n$, i.e. $\gamma=\tilde{g}_{|\partial M}$. Furthermore, the metric $g$ is $C^{n-2, \alpha}$-conformally compact for $n>4$ and $C^{1,\alpha}$ for $n=4$. Biquard \cite{Biq} and Lee \cite{Lee} independently extended this result to boundary metrics in an open $C^{m,\alpha}$-neighborhood of the boundary metric $\gamma_0$ of an arbitrary non-degenerate asymptotically hyperbolic Einstein manifold $(M,g)$. Anderson \cite{AndGAFA} gave other existence results using degree arguments, under the assumption that the boundary metric $\gamma$ has positive scalar curvature.

\subsection{Bryant soliton and its generalizations}
Another class of Riemannian manifolds satisfying the assumptions of Theorem \ref{thm:ExMinMain} is given by Ricci solitons of Bryant type. These metrics have non-negative Ricci curvature and are locally $C^0$-asymptotically Euclidean.  R. Bryant in \cite{Bryant} proved that it is possible to find a function $\phi:\R^+\to \R$ such that the warped product metric
$$g=dr^2+\phi(r)^2 g_{S^{n-1}}$$
on $\R^n$, where $g_{S^{n-1}}$ is the standard metric on $S^{n-1}$ and $r=\sqrt{(x^1)^2+\ldots+(x^n)^2}$ is the radial coordinate, is a complete metric with  positive curvature operator (hence positive Ricci curvature), whose sectional curvatures decay at least inverse linearly in $r$ (Bryant's proof is in dimension three, but analogous arguments give the general case; see for example Section 4.6 in \cite{CLN}). This metric plays a crucial role in the analysis of Ricci flow, being the only example in dimension three of a  non-flat and $k$-non-collapsed steady gradient Ricci soliton (see \cite{Bre3D} and \cite{BreHD} for higher dimension).  

Other soliton examples fitting our assumptions are given by Catino-Mazzieri in \cite{CM12}.

\end{document}